\documentclass[12pt]{amsart}
\usepackage{amsmath, amsthm, amssymb}
\usepackage[top=1.25in, bottom=1.25in, left=1.0in, right=1.0in]{geometry}

\allowdisplaybreaks
\usepackage{tkz-graph}
\usetikzlibrary{arrows}
\usetikzlibrary{shapes}
\usepackage[position=bottom]{subfig}
\usepackage{setspace}
\makeatletter
\newtheorem*{rep@theorem}{\rep@title}
\newcommand{\newreptheorem}[2]{
\newenvironment{rep#1}[1]{
 \def\rep@title{#2 \ref{##1}}
 \begin{rep@theorem}}
 {\end{rep@theorem}}}
\makeatother

\theoremstyle{plain}
\newtheorem{thm}{Theorem}
\newreptheorem{thm}{Theorem}

\newreptheorem{prop}{Proposition}
\newtheorem{lem}[thm]{Lemma}
\newreptheorem{lem}{Lemma}

\newreptheorem{conjecture}{Conjecture}

\newreptheorem{cor}{Corollary}

\theoremstyle{definition}
\newtheorem{defn}{Definition}
\theoremstyle{remark}

\newcommand{\IN}{\mathbb{N}}

\newcommand{\R}{\mathfrak{R}}
\newcommand{\F}{\mathfrak{F}}

\newcommand{\set}[1]{\left\{ #1 \right\}}
\newcommand{\setb}[3]{\left\{ #1 \in #2 \mid #3 \right\}}
\newcommand{\setbs}[2]{\left\{ #1 \mid #2 \right\}}
\newcommand{\card}[1]{\left|#1\right|}

\newcommand{\irange}[1]{\left[#1\right]}

\newcommand{\parens}[1]{\left( #1 \right)}

\newcommand{\DefinedAs}{\mathrel{\mathop:}=}

\title{Structural fixed-point theorems}\thanks{[Editorial Comment, added 2021] This note was originally written sometime in 2007 with various updates in 2008 and 2009, and then updated into its current state in 2014. The authors had planned on expanding it and submitting for publication, but other interests got in the way and it sat dormant. Unfortunately, Landon Rabern died in October 2020, so the authors will not be expanding it. Since the core ideas are interesting the first author has decided to make the paper available online.  See \cite{danger}, appendix A (The Global Function) for an alternative, but much condensed presentation. For more recent frameworks applying graph-theoretic tools to semantic paradoxes see  Beringer, T., \& Schindler, T. (2017). A graph-theoretic analysis of the semantic paradoxes. \textit{The Bulletin of Symbolic Logic}, 23(4), 442-492; Hsiung, M. (2020). What paradoxes depend on. \textit{Synthese}, 197(2), 887-913.}
\author{Brian Rabern and Landon Rabern}

\begin{document}
\maketitle
\section{Introduction}
\noindent The semantic paradoxes are often associated with self-reference or referential circularity. However, Yablo has shown in \cite{yablo93} that there are infinitary versions of the paradoxes that do not involve this form of circularity. It remains an open question what \textit{relations of reference} between collections of sentences afford the structure necessary for paradoxicality. In \cite{danger} we laid the groundwork for a general investigation into the nature of reference structures that support the semantic paradoxes.  The remaining task is to classify the so-called \emph{dangerous} directed graphs.  In appendix A of \cite{danger}, we sketched a reformulation of the problem in terms of fixed points of certain functions.  Here we expand on this reformulation, removing all syntactic considerations to get a purely mathematical problem.  It is definitely possible that the problem's solution depends on the axioms of set theory we choose---this would be an interesting outcome.

For sets $A$ and $B$, we write $A^B$ for all functions from $B$ to $A$. We use standard von Neumann ordinals, so $2 \DefinedAs \set{0,1}$, etc.  When it is convenient to visualize $g \in A^B$ as a sequence of elements from $A$ indexed by $B$, we use the notation $g_v$ to mean $g(v)$ for $v \in B$. 

\section{Dangerous graphs}

\noindent A \emph{directed graph} is a pair $G \DefinedAs (V, E)$ where $V$ is any set and $E \subseteq V^2$. For all $\alpha \in 2^V$ and $I \subseteq V$, put

\[\alpha^I \DefinedAs \setb{\beta}{2^V}{\forall_{x \in V - I} \beta(x) = \alpha(x)}.\]

\vspace{.1in}
\noindent A function $g \in 2^{2^V}$ is \emph{independent of} $I \subseteq V$ if $g$ is constant on $\alpha^I$ for each $\alpha \in 2^V$. A function $f \in \parens{2^V}^{2^V}$ \emph{respects} $G$ if for each $v \in V$, the function $f^v$ given by $f^v(x) \DefinedAs f(x)_v$ is independent of $V - N^+(v)$.  We let $\R(G)$ be all the $f \in \parens{2^V}^{2^V}$ respecting $G$.

\begin{defn}
A graph $G$ is \emph{dangerous} if $\R(G)$ contains a fixed-point-free function.
\end{defn}

It is not difficult to see that on this definition, the same graphs come out as dangerous as those in \cite{danger}. Define an inverse operation for $f \in \parens{2^V}^{2^V}$ by $\R^{-1}(f) \DefinedAs \setbs{G}{f \in \R(G)}$.  Let $\F \subseteq \parens{2^V}^{2^V}$ be the fixed-point-free functions.  Then $\R^{-1}(\F)$ is the set of all dangerous graphs (on $V$).  It is easy to see that a graph is dangerous if any subgraph of it is and thus the complement of $\R^{-1}(\F)$ can be defined by excluding subgraphs (excluding all of $\R^{-1}(\F)$ is one such description).  The problem is to find such a description using a ``small'' set of subgraphs.  When $V$ is finite, the complement of $\R^{-1}(\F)$ is given precisely by excluding the set of all directed cycles (on $V$).  The next case to figure out is when $V = \omega$.  

\section{An example: locally finite graphs}
\noindent In \cite{danger}, the dangerous locally finite graphs we're classified using G\"{o}del compactness.  In fact, a larger class of dangerous graphs was classified by a direct application of Zorn's lemma.  This same work can be done in the setting of fixed points of functions using Tychonoff's theorem.  This section gives an example of this method.

The graphs in this section all have $V = 2^\omega$. Adorn $2^\omega$ with the product topology where each copy of $2$ has the discrete topology.  By Tychonoff's theorem, $2^\omega$ is compact.  Also, this topology is generated by the $p$-adic metric on $2^\omega$; that is, for sequences $(x_i)_{i \in \omega}$ and $(y_i)_{i \in \omega}$ define $d((x_i), (y_i))$ as $\frac{1}{k}$ where $k \in \omega$ is minimal such that $x_k \neq y_k$.  Using this metric, the following is easy to see. 

We call a graph $G$ \emph{locally finite} if $\card{N^+(v)}$ is finite for all $v \in V(G)$.

\begin{lem}\label{LocallyFiniteImpliesContinuous}
All functions in $\R(G)$ are continuous iff $G$ is locally finite.
\end{lem}

We'll need a simple fixed point lemma.

\begin{lem}\label{SimpleFixedPoint}
Let $(M,d)$ be a compact metric space.  If $f \in M^M$ is continuous and for every $\epsilon > 0$ there exists $x \in M$ such that $d(f(x), x) < \epsilon$, then $f$ has a fixed point.
\end{lem}
\begin{proof}
Let $f \in M^M$ be such a function.  For $k \in \IN$, pick $x_k \in M$ such that $d(f(x_k), x_k) < \frac{1}{k}$.  Since $M$ is compact, the sequence $x_1, x_2, \ldots$ has a convergent subsequence $x_{n_1}, x_{n_2}, \ldots$ converging to $x$, say.  Since $f$ is continuous, the sequence $f(x_{n_1}), f(x_{n_2}), \ldots$ converges to $f(x)$. We claim that $f(x) = x$.  Suppose not, then we have $\epsilon >0$ such that $d(f(x), x) > \epsilon$.  Since these two sequences converge, we can choose $k \in \IN$ so that $\max\set{\frac{1}{n_k}, d(x_k, x), d(f(x), f(x_k))} < \frac{\epsilon}{3}$.  But then we have $d(f(x), x) \leq d(f(x), f(x_k)) + d(f(x_k), x) \leq d(f(x), f(x_k)) + d(f(x_k), x_k) + d(x_k, x) < \epsilon$, a contradiction.
\end{proof}

Now the characterization of dangerous graphs follows easily.

\begin{thm}\label{LocallyFinite}
A locally finite graph is dangerous iff it contains a directed cycle.
\end{thm}
\begin{proof}
Plainly, if a graph contains a directed cycle it is dangerous.  For the other direction, let $G$ be an acyclic locally finite graph and $f \in \R(G)$.  By Lemma \ref{LocallyFiniteImpliesContinuous}, $f$ is continuous. Since acyclic finite graphs are not dangerous, for each $k \in \omega$ and any choice of $x_{k+1}, x_{k+2}, x_{k+3}, \ldots$ we have $x \in 2^\omega$ such that $f(x)_i = x_i$ for $i \in \irange{k}$.  Therefore, in the $p$-adic metric, $d(f(x), x) < \frac{1}{k}$.  Now applying Lemma \ref{SimpleFixedPoint} shows that $f$ has a fixed point.  Hence $G$ is not dangerous.
\end{proof}

We note that by using ultrafilter convergence, the use of metric spaces can be done away with and Theorem \ref{LocallyFinite} can be proved in a similar way for all $V$.  This more general result is proved using G\"{o}del compactness in \cite{danger}.

\bibliographystyle{amsplain}
\bibliography{SPRG}
\end{document}